\documentclass[11pt]{article}
\usepackage[english]{babel}
\usepackage{latexsym,amsfonts,amsmath,enumerate,graphics,enumerate,amsthm,tikz,hyperref}
\usepackage[affil-it]{authblk}
\usepackage{enumerate,amsthm,dsfont,pstricks}
\usepackage{latexsym,amsfonts,amsmath,amssymb,amscd,wrapfig,graphicx,curves}
\usepackage[a4paper,margin=2.5cm]{geometry}

\newcommand{\inprod}[1]{\left\langle #1 \right\rangle}

\newcommand{\leaderfill}{\leaders\hbox to 1em{\hss.\hss}\hfill}
\newcommand{\CC}{\mathbb{C}}

\newcommand{\aff}{\mathrm{aff}\,}
\newcommand{\Span}{\mathrm{span}\,}

\newtheorem{lemma}{Lemma}[section]

\newtheorem{theorem}[lemma]{Theorem}

\newtheorem{corollary}[lemma]{Corollary}

\numberwithin{equation}{section}

\let\phi=\varphi
\let\theta=\vartheta

\makeatletter
\def\@maketitle{%
  \newpage
  \null
  \vskip 2em%
  \begin{center}%
  \let \footnote \thanks
    {\Large\bfseries \@title \par}%
    \vskip 1.5em%
    {\normalsize
      \lineskip .5em%
      \begin{tabular}[t]{c}%
        \@author
      \end{tabular}\par}%
    \vskip 1em%
    {\normalsize \@date}%
  \end{center}%
  \par
  \vskip 1.5em}
\makeatother

\begin{document}

\title{\bf \huge Midpoints  for Thompson's metric on symmetric cones}

\author{Bas Lemmens and Mark Roelands%
\thanks{Mark Roelands was supported by EPSRC grant EP/J500446/1}}

\maketitle
\date{}

\begin{abstract} 
We  characterise the affine span of the midpoints sets, $\mathcal{M}(x,y)$, for Thompson's metric on  symmetric cones in terms of a translation of the zero-component of the Peirce decomposition of an idempotent. As a consequence we derive an explicit formula for the dimension of the affine span of $\mathcal{M}(x,y)$ in case the associated Euclidean Jordan algebra is simple. In particular, we find for $A$ and $B$ in the cone positive definite Hermitian matrices that 
\[
\dim(\aff \mathcal{M}(A,B))=k^2,
\]
where $k$ is the number of eigenvalues $\mu$ of $A^{-1}B$, counting multiplicities, such that \[
\mu\neq \max\{\lambda_+(A^{-1}B),\lambda_-(A^{-1}B)^{-1}\},
\] where 
$\lambda_+(A^{-1}B):=\max \{\lambda\colon \lambda\in\sigma(A^{-1}B)\}$ and 
$\lambda_-(A^{-1}B):=\min\{\lambda\colon \lambda\in\sigma(A^{-1}B)\}$. 
These results extend work by Y. Lim \cite{Lim3}. 
\end{abstract}

\section{Introduction}
The space of $n\times n$ Hermitian matrices contains a cone $\Pi_n(\CC)$ of all positive-semidefinite matrices. Its  interior,  $\Pi_n(\CC)^\circ$, consists  of all invertible elements, and is a prime example of a symmetric cone. It is well know, see for example \cite{Bha}, that $\Pi_n(\CC)^\circ$ can be equipped with a Riemannian metric
\[
\delta_2(A,B):=\|\log(A^{-1}B)\|_2=\left(\sum_{i=1}^n (\log\lambda_i(A^{-1}B))^2\right)^{1/2},
\] 
where the $\lambda_i(A^{-1}B)$'s are the eigenvalues of $A^{-1}B$. The metric space $(\Pi_n(\CC)^\circ,\delta_2)$ is geodesic, i.e., any two points are connected by a geodesic. In fact, in this case the geodesic is unique and given by the path, 
\[
t\mapsto A^{1/2}(A^{-1/2} BA^{-1/2})^tA^{1/2},
\]
where $t\in [0,1]$, and the {\em geometric mean} 
\[
A\sharp B:= A^{1/2}(A^{-1/2} BA^{-1/2})^{1/2}A^{1/2}
\]
is the unique midpoint of $A$ and $B$. 

Another natural metric on $\Pi_n(\CC)^\circ$ is {\em Thompson's metric}, 
\[
d_T(A,B):= \|\log(A^{-1}B)\|_\infty = \max_i\left|\log\lambda_i(A^{-1}B)\right|.
\]
The space $(\Pi_n(\CC)^\circ,d_T)$ is a geodesic Finsler  metric space, see \cite{Lim1,Nu}, in which the path $t\mapsto A^{1/2}(A^{-1/2} BA^{-1/2})^tA^{1/2}$ is also a geodesic, but in general not unique. 

Thompson's metric, which was introduced in \cite{Tho}, is a useful metric that can be defined on the interior of any closed cone in a normed space. It is widely applied in the study of operator means \cite{LL2,LL3,LeeL,LP,PP}, geometric analysis of spaces of operators \cite{ACS,CM,CPR,LL1,Mo}, and the spectral theory of linear and nonlinear operators on cones \cite{AGLN,GQ,HIR,LNBook,LLNW,Nmem1,Tho}. In \cite{Lim3} Lim studied the geometry of the {\em midpoints set} 
\[
\mathcal{M}(A,B):=\left\{C\in\Pi_n(\CC)^\circ\colon d_T(A,C)=\frac{1}{2}d_T(A,B)=d_T(C,B)\right\}
\] 
for $A,B\in\Pi_n(\CC)^\circ$. Among other results Lim \cite[Theorem 5.2]{Lim3} showed that $\mathcal{M}(A,B)$ is a singleton if, and only if, $\sigma(A^{-1}B)\subseteq \{\alpha,\alpha^{-1}\}$ for some $\alpha>0$. This result has been generalised by the authors in \cite{LR} to the cone of positive self-adjoint elements in a unital $C^*$ algebra and symmetric cones.  

In general cones the  midpoints set $\mathcal{M}(x,y):=\{z\in K^\circ\colon d_T(x,z)=\frac{1}{2}d_T(x,y)=d_T(z,y)\}$ is convex, as it is the intersection of the Thompson metric balls $B(x,1/2)$ and $B(y,1/2)$, which are both convex, see \cite[Lemma 2.6.2]{LNBook}. 
The main goal of this  paper is to characterise the affine span of the midpoints set $\mathcal{M}(x,y)$ for $x$ and $y$ in a symmetric cone in terms of a translation of the zero-component of the Peirce decomposition of an idempotent. As a corollary we obtain an explicit formula for the dimension of the affine span of $\mathcal{M}(x,y)$ in case the associated Euclidean Jordan algebra is simple. In the special case where $A,B\in\Pi_n(\CC)^\circ$  we find that 
\[
\dim(\aff \mathcal{M}(A,B))=k^2,
\]
where $k$ is the number of eigenvalues $\mu$ of $A^{-1}B$, counting multiplicities, such that 
\[\mu\neq \max\{\lambda_+(A^{-1}B),\lambda_-(A^{-1}B)^{-1}\}\] 
and 
$\lambda_+(A^{-1}B):=\max \{\lambda\colon \lambda\in\sigma(A^{-1}B)\}$ and 
$\lambda_-(A^{-1}B):=\min\{\lambda\colon \lambda\in\sigma(A^{-1}B)\}$. 

To obtain the results we first prove a characterisation of the midpoints set in  a general cone in terms of its faces, see Theorem \ref{thm:3.2}. This result is subsequently used  in Section 4 to find the affine span of the midpoints set, and its  dimension, in symmetric cones. We begin by recalling some basic definitions.  

\section{Preliminaries}
A cone $K$ in a vector space $V$ is a convex subset such that $K\cap(-K)=\{0\}$ and $\lambda K\subseteq K$ for all $\lambda\ge 0$. It induces a partial ordering $\le_K$ on $V$ by putting $x\le_K y$ if $y-x\in K$. Given $x\leq_K y$ in $V$ we denote the {\em order interval} by $[x,y]_K:=\{z\in V\colon x\leq_K z\leq_K y\}$. A non-empty convex subset $F\subseteq K$ is said to be a \textit{face} of $K$ if $x,y\in K$ such that $\lambda x+(1-\lambda)y\in F$ for some $0<\lambda<1$ implies that $x,y\in F$. The face generated by $z\in K$ is denoted $F_z$, i.e.,  
\[ 
F_z:=\{y\in K\colon \lambda y +(1-\lambda)z\in K\mbox{ for some }\lambda < 0\}.
\]

We say that a cone $K$ is \textit{Archimedean} if for all $x\in V$ and $y\in K$ with $nx\le_K y$ for $n\ge 1$ we have that $x\le_K 0$. An element $u\in K$ is called an \textit{order unit} if for each $x\in V$ there exists $\lambda>0$ such that $x\le_K\lambda u$. The triple $(V,K,u)$ is called an \textit{order unit space} if $K$ is Archimedean and $u$ is an order unit.

In an order unit space the faces can be characterised as follows. 
\begin{lemma}\label{face} If $(V,K,u)$ is an order unit space and $z\in K$, then 
 $$F_z=\bigcup_{n\ge 1}[0,nz]_K.$$
\end{lemma}
\begin{proof} Note that $\lambda y +(1-\lambda)z\in K$ for some $\lambda < 0$ if and only if $-\mu y+(1+\mu)z\in K$ for some $\mu >0$, which is equivalent to $y\leq_K \alpha z\in K$ for some $\alpha>0$. Thus, 
$F_z=\cup_{\alpha\geq 1} [0,\alpha z]_K =\cup_{n\geq 1} [0,nz]_K$.
 \end{proof}

An order unit space $(V,K,u)$  can be equipped  with the \textit{order unit norm} $\|\cdot\|_u$, which is  defined by
$$\|x\|_u:=\inf\{\lambda>0:-\lambda u\le_K x\le_K \lambda u\}.$$
With respect to this norm, the cone $K$ is closed by \cite[Theorem 2.55(2)]{AT}. Furthermore, the norm $\|\cdot\|_u$ is \textit{monotone}, that is, $\|x\|_u\le\|y\|_u$ for all  $0\le_K x\le_K y$. In particular, $K$ is a {\em normal} cone with respect to $\|\cdot\|_u$, i.e., there exists a constant $\kappa>0$ such that $\|x\|_u\leq \kappa\|y\|_u$ whenever $x\leq_K y$ in $V$. It is known, see \cite[Lemma 2.5]{AT}, that each interior point of $K$ is an order unit of $K$. On the other hand, if $x\in K $ is an order unit of $K$, then there exists $M>0$ such that $u\leq_K Mx$. So, for $y\in V$ with $\|y\|_u\leq 1/M$ we have that $0\leq_K x-u/M\leq_K x-y$, which show that $x\in K^\circ$. Thus, the interior $K^\circ$ coincides with the set of order units of $K$.

We see that given an order unit space $(V,K,u)$ and $x,y\in K^\circ$, there are constants $0<\beta$ such that $ x\le_K\beta y$, and hence we can define
$$M(x/y):=\inf\{\beta>0: x\le_K\beta y\}<\infty.$$
Now \textit{Thompson's metric} on $K^\circ$ is given by
$$d_T(x,y):=\log(\max\{M(x/y),M(y/x)\})$$
and was introduced  in \cite{Tho}.
The set of midpoints is denoted 
\[
\mathcal{M}(x,y):=\left\{z\in K^\circ\colon d_T(x,z)=\frac{1}{2}d_T(x,y)=d_T(z,y)\right\}.
\]

Thompson's metric spaces $(K^\circ, d_T)$ are geodesic spaces, see \cite[Section 2.6]{LNBook}, that is to say, any two points in $K^\circ$ are connected by a geodesic segment. Recall that a map $\gamma$ from an (open, closed, bounded, or, unbounded) interval $I\subseteq\mathbb{R}$ into a metric space $(X,d)$ is called a \textit{geodesic path} if $d(\gamma(s),\gamma(t))=|s-t|$ for all $s,t\in I$. The image of $\gamma$ is called a \textit{geodesic segment} in $(X,d)$, and for $x,y\in(X,d)$. 

\section{Midpoints  in general cones}
Before we give the characterisation of the midpoints set in $(K^\circ,d_T)$, where $(V,K,u)$ is an order unit space, we make some preliminary observations. 
To begin, we note that if $x,y\in K^\circ$ are linearly dependent, then the straight-line segment connecting $x$ and $y$  is a unique geodesic segment for Thompson's metric, see \cite[Lemma 3.3]{LR}, and hence the midpoints set is a singleton in that case. So, in the sequel we only need to consider the midpoints sets of  linearly independent elements of $K^\circ$. 

If $x,y\in K^\circ$ are linearly independent, then we write $V(x,y)=\Span (x,y)$ and we let  $K(x,y):=V(x,y)\cap K$ be the 2-dimensional cone containing $x$ and $y$, which has  relative interior $K(x,y)^\circ$ in $V(x,y)$. 
Note that for $w,z\in K(x,y)^\circ$ the distance $d_T(w,z)$ with respect to $K(x,y)$ is the same  as $d_T(w,z)$ with respect to $K$. As $K(x,y)$ is a closed cone in $V(x,y)$ and $K(x,y)^\circ$ is non-empty, we know \cite[Theorem A.5.1]{LNBook} that there exist  linearly independent linear functionals $\psi_1$ and $\psi_2$ on $V(x,y)$ such that 
\[
K(x,y)=\{z\in V(x,y)\colon \psi_1(z)\geq 0\mbox{ and }\psi_2(z)\geq 0\}.
\]
The linear map $\Psi\colon V(x,y)\to\mathbb{R}^2$ given by, $\Psi(z) =(\psi_1(z),\psi_2(z))$ for $z\in V(x,y)$,  maps $K(x,y)$ onto the standard positive cone $\mathbb{R}^2_+:=\{(w_1,w_2)\colon w_1\geq 0\mbox{ and }w_2\geq 0\}$. Furthermore for $u,v\in K(x,y)^\circ$ we have that 
$M(u/v)=M(\Psi(u)/\Psi(v))$, and hence $\Psi$ is a $d_T$-isometry. One can verify that in $((\mathbb{R}^2_+)^\circ, d_T)$ the path $t\mapsto \Psi(x)^{1-t}\Psi(y)^t$, for $t\in [0,1]$,  is a geodesic path from $\Psi(x)$ to $\Psi(y)$. The pull--back of  this geodesic  path under the isometry $\Psi$ is an geodesic path connecting $x$ and $y$ in $(K(x,y)^\circ, d_T)$. We will call it the {\em canonical geodesic} connecting $x$ and $y$ and  denote it by $\gamma_{xy}$. Moreover, the midpoint 
\[
m_{xy}:=\gamma_{xy}(1/2)
\] 
is said to be the {\em canonical midpoint} of $x$ and $y$. 
Note that 
\[
M(x/m_{xy}) = M(\Psi(x)/\Psi(x)^{1/2}\Psi(y)^{1/2}) = M(\Psi(x)^{1/2}/\Psi(y)^{1/2}) =M(x/y)^{1/2},
\]
so that $M(x/y)=M(x/m_{xy})^2$. Likewise it can be shown that 
\begin{equation}\label{eq:3.1}
M(m_{xy}/y)^2=M(x/y)\mbox{\quad and \quad} M(m_{xy}/x)^2=M(y/x)=M(y/m_{xy})^2.
\end{equation}

Given $x,y\in K^\circ$, we let $\ell_{xy}^+:=\{\lambda y+(1-\lambda)x\colon \lambda\geq 0\}$ be the half-line emanating from $x$ through $y$. 
The following basic observation will be useful. 
\begin{lemma}\label{lem:3.1}
Let $(V,K,u)$ be an order unit space and $x,y\in K^\circ$. If $M:=M(x/y)>1$, then $\ell_{xy}^+$ intersects $\partial K$ in 
\[
y':= \frac{M}{M-1}y+\frac{1}{1-M}x.
\]
\end{lemma}
\begin{proof}
Note that, as $K$ is closed in $(V,K,u)$, we have that $y-M^{-1}x\in\partial K$. This implies that $y':= \frac{M}{M-1}y+\frac{1}{1-M}x\in\partial K$. As $y'$ is also on $\ell_{xy}^+$ the result follows. 
\end{proof}
For a non-empty subset $W\subseteq V$ we will denote the affine span of $W$ in $V$ by $\aff W$. 
\begin{theorem}\label{thm:3.2} Let $(V,K,u)$ be an order unit space and $x,y\in K^\circ$ be linearly independent. Then the affine span of $\mathcal{M}(x,y)$ satisfies: 
\begin{enumerate}[(i)]
\item $\aff\mathcal{M}(x,y)=m_{xy}+\Span F_{y'}$, if $M(x/y)>M(y/x)$;
\item $\aff\mathcal{M}(x,y)=m_{xy}+\Span F_{x'}$, if $M(y/x)>M(x/y)$;
\item $\aff\mathcal{M}(x,y)=m_{xy}+\Span F_{x'}\cap\Span F_{y'}$, if $M(x/y)=M(y/x)$.
\end{enumerate}
\end{theorem}
\begin{proof} Note that case (ii) follows from  case (i) by symmetry.  So, suppose that $M(x/y)>M(y/x)$ and write $M:=M(x/y)$. As $d_T(x,y)=\log M$, we see that $M>1$.  It now follows from Lemma \ref{lem:3.1} that $\ell_{xy}^+$ intersects $\partial K$ in  
\[
y':=\frac{M}{M-1}y+\frac{1}{1-M}x.
\] 

Let $z\in\mathcal{M}(x,y)$. We deduce from
\[
d_T(x,y) \leq \log M(x/z)+\log M(z/y)\leq d_T(x,z)+d_T(z,y)=d_T(x,y)
\]
that $d_T(x,z)=\log M(x/z)=\log M(z/y)=d_T(z,y)$, so that $M(x/z)^2=M(z/y)^2=M$. Write $N:=M(x/z)>1$. 
Again using Lemma \ref{lem:3.1}  the half-line $\ell_{xz}^+$ intersects $\partial K$ in
\[
z_1:=\frac{N}{N-1}z+\frac{1}{1-N}x,
\]
and $\ell_{zy}^+$  intersects $\partial K$ in
\[
z_2:=\frac{N}{N-1}y+\frac{1}{1-N}z,
\]
see Figure \ref{Fig:1}.
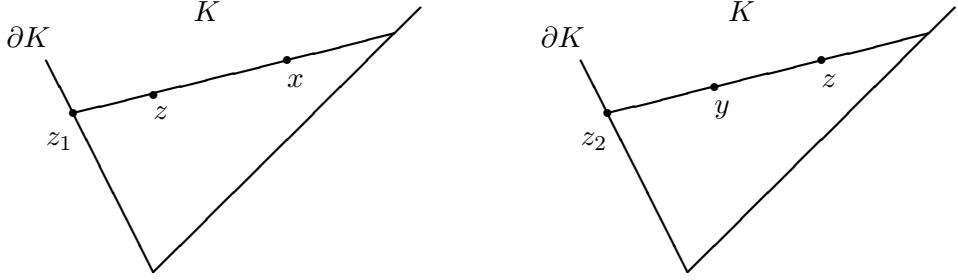
\begin{figure}[h]
\begin{center}
\thicklines
\begin{picture}(350,110)(0,0)
\put(50,0){\line(-1,2){40}}
\put(50,0){\line(1,1){100}}
\put(20,60){\line(4,1){120}}
\put(-5,85){$\partial K$}
\put(65,95){$K$}
\put(50,67){\circle*{3.0}}\put(50,58){$z$}
\put(100,80){\circle*{3.0}}\put(100,70){$x$}
\put(20,60){\circle*{3.0}}\put(10,48){$z_1$}
\put(250,0){\line(-1,2){40}}
\put(250,0){\line(1,1){100}}
\put(220,60){\line(4,1){120}}
\put(195,85){$\partial K$}
\put(265,95){$K$}
\put(260,70){\circle*{3.0}}\put(260,60){$y$}
\put(300,80){\circle*{3.0}}\put(300,70){$z$}
\put(220,60){\circle*{3.0}}\put(210,48){$z_2$}
\end{picture}
\caption{Endpoints \label{Fig:1}}
\end{center}
\end{figure}

Working out the following convex combination
\begin{align*}\frac{1}{N+1}z_1+\frac{N}{N+1}z_2&=\left(\frac{N}{N^2-1}z+\frac{1}{1-N^2}x\right)+\left(\frac{N^2}{N^2-1}y+\frac{N}{1-N^2}z\right)\\&=\frac{N^2}{N^2-1}y+\frac{1}{1-N^2}x=\frac{M}{M-1}y+\frac{1}{1-M}x\\&=y'
\end{align*}
shows that $z_1$ and $z_2$ both belong to the face $F_{y'}$.

Write $z:=m_{xy}+v$ for some $v\in V$. Let $z_3\in F_{y'}$ be the points of intersection of $\ell_{m_{xy}y}^+$ and $\partial K$. So, 
\[
z_3:=\frac{N}{N-1}y+\frac{1}{1-N}m_{xy}
\] 
by Lemma \ref{lem:3.1}. It follows that $v=(N-1)(z_3-z_2)\in \Span F_{y'}$ which yields the inclusion $$\aff\mathcal{M}(x,y)\subseteq m_{xy}+\Span F_{y'}.$$

Conversely, suppose $v\in\Span F_{y'}$, with $v\neq 0$.  Define $z:=m_{xy}+v$.  By Lemma \ref{lem:3.1} the point 
\[
m_{xy}':= \frac{N}{N-1} y+\frac{1}{1-N}m_{xy}
\]
lies in $\partial K$. As $\gamma_{xy}$ lies in $K(x,y)$, we see that $m_{xy}$ is a positive scalar multiple of $y'$, and hence $m_{xy}'$ lies in the relative interior of $F_{y'}$. 
Let $t = (1-N)^{-1}$ and note that 
\[
m_{xy}' +tv = \frac{N}{N-1}y + \frac{1}{1-N}(m_{xy} +v).
\]

As $m_{xy}'$ is in the relative interior of $F_{y'}$, we can replace $v$ by $\epsilon v$ for some $\epsilon>0$ sufficiently small, and assume that $m_{xy}'+tv\in F_{y'}$ and $m_{xy}+v\in K^\circ$. We know from \cite[pp.\,28--29]{LNBook} that 
\[
M(m_{xy}/y) = \frac{|m_{xy}m_{xy}'|}{|ym_{xy}'|}\mbox{\quad and \quad }
M(z/y) = \frac{|z(m'_{xy}+tv)|}{|y(m_{xy}' +tv)|},
\]
where $|uw'|/|ww'|$ denotes the ratio of the lengths of the straight-line segments $[u,w'$] and $[w,w']$, see Figure \ref{Fig:2}. 
\begin{figure}[h]
\begin{center}
\thicklines
\begin{picture}(160,130)(0,0)
\put(50,20){\line(-1,3){35}}
\put(0,100){\line(1,0){190}}
\put(15,50){\line(2,1){165}}
\put(23,100){\circle*{3,0}}\put(-28,107){$m_{xy}'+tv$}
\put(36,61){\circle*{3,0}}\put(12,67){$m_{xy}'$}
\put(115,100){\circle*{3,0}}\put(111,109){$y$}
\put(170,100){\line(-1,3){8}}
\put(170,100){\circle*{3,0}}\put(165,88){$z$}
\put(162,124){\circle*{3,0}}\put(150,132){$m_{xy}$}
\put(53,10){$F_{y'}$}
\end{picture}
\caption{The endpoints in $F_{y'}$\label{Fig:2}}
\end{center}
\end{figure}
Using similarity of triangles we conclude that $M(m_{xy}/y) =M(z/y)$. 

Similarly, let 
\[
\tilde{m}_{xy} = \frac{N}{N-1} m_{xy}+ \frac{1}{1-N} x.
\] 
Note that $\tilde{m}_{xy}\in \partial K$ by Lemma \ref{lem:3.1} and is a positive scalar multiple of $y'$, as $\gamma_{xy}$ is contained in $K(x,y)$. By possibly further reducing $\epsilon >0$ we may assume for $s:= N/(N-1)$ that
\[
\tilde{m}_{xy} +sv = \frac{N}{N-1}(m_{xy}+v) +\frac{1}{1-N}x\in F_{y'}
\]
and $m_{xy}+v\in K^\circ$. Again using similarity of triangles, see Figure 
\ref{Fig:3}, we get that 
\[
M(x/m_{xy}) = \frac{|x\tilde{m}_{xy}|}{|m_{xy}\tilde{m}_{xy}| }
=\frac{|x(\tilde{m}_{xy}+sv)|}{|z(\tilde{m}_{xy}+sv)|}=M(x/z).
\]
\begin{figure}[h]
\begin{center}
\thicklines
\begin{picture}(160,120)(0,0)
\put(50,20){\line(-1,3){33}}
\put(0,100){\line(1,0){190}}
\put(30,30){\line(2,1){158}}
\put(23,100){\circle*{3,0}}\put(-2,108){$\tilde{m}_{xy}$}
\put(115,100){\circle*{3,0}}\put(108,108){$m_{xy}$}
\put(170,100){\circle*{3,0}}\put(170,88){$x$}
\put(53,10){$F_{y'}$}
\put(44,37){\circle*{3,0}}\put(-6,38){$\tilde{m}_{xy}+tv$}
\put(123,77){\line(-1,3){8}}\put(123,77){\circle*{3,0}}
\put(128,66){$z$}
\end{picture}
\caption{The endpoints in $F_{y'}$\label{Fig:3}}
\end{center}
\end{figure}

It now follows from (\ref{eq:3.1}) that 
 \[
 M(x/z)^2 = M(x/m_{xy})^2 = M(x/y)  > M(y/x)=M(m_{xy}/x)^2.
 \]
 As the map $(u,w)\mapsto M(u/w)$ is continuous on $K^\circ\times K^\circ$, see \cite[Lemma 2.2]{LLNW}, we can assume, after possibly further reducing $\epsilon >0$, that $M(z/x)^2<M(x/y)$. It now follows that  $d_T(x,z) = d_T(x,m_{xy})=\frac{1}{2}d_T(x,y)$. In the same it can be shown that $d_T(z,y) = d_T(m_{xy},y)=\frac{1}{2}d_T(x,y)$. We conclude that $z\in \mathcal{M}(x,y)$ and hence $m_{xy}+ \Span F_{y'} \subseteq \aff \mathcal{M}(x,y)$.

Finally, suppose that $M(x/y)=M(y/x)$. We have already shown that the inclusions $\mathrm{aff}\mathcal{M}(x,y)\subseteq m_{xy}+\Span F_{x'}$ and $\mathrm{aff}\mathcal{M}(x,y)\subseteq m_{xy}+\Span F_{y'}$ hold, which immediately implies that 
\[
\mathrm{aff}\mathcal{M}(x,y)\subseteq m_{xy}+\Span F_{x'}\cap\Span F_{y'}.
\] 
Moreover, if $v\in\Span F_{x'}\cap\Span F_{y'}$ and $z:= m_{xy}+\epsilon v$, then  we have also shown that for small enough $\epsilon >0$, the equalities $\log M(x/z)=\log M(z/y)=\frac{1}{2}d_T(x,y)$ hold. Now since $M(x/y)=M(y/x)$, we can apply the same argument to show that 
$\log M(z/x)=\log M(y/z)=\frac{1}{2}d_T(x,y)$, and hence  
\[
m_{xy}+\Span F_{x'}\cap\Span F_{y'}\subseteq\aff\mathcal{M}(x,y),
\]
which proves the last assertion.
\end{proof}

\section{Midpoints in symmetric cones}
The interior $K^\circ$ of a closed cone $K$ in a finite-dimensional inner-product space $(V,\langle\cdot,\cdot\rangle)$ is called a \textit{symmetric cone} if the \textit{dual cone}, $K^*:=\{y\in V\colon \langle y,x\rangle\geq 0\mbox{ for all }x\in K\}$ satisfies  $K^*=K$, and the automorphism group $\mathrm{Aut}(K):=\{A\in\mathrm{GL}(V)\colon A(K)=K\}$ acts transitively on $K^\circ$. A prime example is the cone of positive definite Hermitian matrices.

It is well known that the symmetric cones in finite dimensions are precisely the interiors of  the cones of squares of Euclidean Jordan algebras.  We will follow the notation and terminology from \cite{FK}, which gives detailed account of the theory of symmetric cones.  

 A \textit{Euclidean Jordan algebra} is a finite-dimensional real inner-product space $(V,\langle\cdot,\cdot\rangle)$ equipped with a bilinear product $(x,y)\mapsto x\bullet y$ from $V\times V$ into $V$ such that for each $x,y\in V$:
\begin{enumerate}
\item[$(i)$] $x\bullet y=y\bullet x$,
\item[$(ii)$] $x\bullet (x^2\bullet y)=x^2\bullet (x\bullet y)$, and
\item[$(iii)$] for each $x\in V$, the linear map $L(x)\colon V\to V$ given by $L(x)y:=x\bullet y$
satisfies
\[
\langle L(x)y, z\rangle = \langle y, L(x)z\rangle\mbox{\quad for all }y,z\in V.
\]
\end{enumerate}
A Euclidean Jordan algebra is not  associative in general, but it is  commutative.
The {\em unit} in a Euclidean Jordan algebra  is denoted by $e$.
An element $c\in V$ is called an {\em idempotent}  if $c^2=c$.
 A set $\{c_1,\ldots,c_k\}$ is called a {\em complete system of orthogonal idempotents} if
 \begin{enumerate}
\item[$(i)$]  $c_i^2= c_i$  for all $i$,
\item[$(ii)$] $c_i\bullet c_j =0$ for all $i\neq j$, and
\item[$(iii)$] $c_1+\cdots+ c_k =e$.
\end{enumerate}

The spectral theorem \cite[Theorem III.1.1]{FK} says that for each $x\in V$ there exist unique real numbers $\lambda_1, \ldots,\lambda_k$, all distinct, and a complete system of orthogonal idempotents $c_1,\ldots,c_k$ such that
$x = \lambda_1 c_1+\cdots +\lambda_k c_k$.
The numbers $\lambda_i$ are called the {\em eigenvalues} of $x$. The {\em spectrum} \index{spectrum} of $x$ is denoted  by
$\sigma(x) =\{\lambda \colon \lambda \mbox{ eigenvalue of } x\}$,
and we write
\[\lambda_+(x)=\max\{\lambda\colon\lambda\in\sigma(x)\} \mbox{\quad and\quad }
\lambda_-(x)=\min\{\lambda\colon\lambda\in\sigma(x)\}.
\]

For $x\in V$   the linear mapping, $P(x) = 2L(x)^2 -L(x^2)$,
is called the {\em quadratic representation} of $x$. Note that $P(x^{-1/2})x=e$ for all $x\in K^\circ$. It is known that $P(x^{-1}) = P(x)^{-1}$ for all $x\in K^\circ$ and  $P(x)\in\mathrm{Aut}(K)$ whenever $x\in K^\circ$, see \cite[Proposition III.2.2]{FK}.
So, $P(x)$ is an isometry of $(K^\circ,d_T)$ if $x\in K^\circ$ by  \cite[Corollary 2.1.4]{LNBook}.
For $x,y\in K^\circ$ we write
\[
\lambda_+(x,y) =\lambda_+(P(y^{-1/2}) x)\mbox{\quad and\quad }
\lambda_-(x,y) =\lambda_-(P(y^{-1/2}) x).
\]
Note that for $x,y\in K^\circ$,  $x\leq_K \beta y$ if and only if $0\leq_K \beta e -P(y^{-1/2})x$, and hence
\[
M(x/y) = \lambda_+(x,y).\]
Similarly, $\alpha y\leq_K x$ is equivalent with $0\leq_K P(y^{-1/2})x-\alpha e$, and hence
\[
M(y/x)^{-1} = \lambda_-(x,y).
\]
So, for $x,y\in K^\circ$ the Thompson metric distance is given by
\[
d_T(x,y) = \log\left( \max\{\lambda_+(x,y), \lambda_-(x,y)^{-1}\} \right).
\]

For $A,B\in \Pi_n(\mathbb{C})^\circ$ we have that $P(B^{-\frac{1}{2}})A=B^{-\frac{1}{2}}AB^{-\frac{1}{2}}$; so, in that  case 
\[
d_T(A,B)=\max\left\{\max_i\log\lambda_i(B^{-\frac{1}{2}}AB^{-\frac{1}{2}}), \max_i-\log\lambda_i(B^{-\frac{1}{2}}AB^{-\frac{1}{2}})\right\} =\max_i\left|\log\lambda_i(B^{-1}A)\right|.
\]

The quadratic representation $P(y^{-1/2})$ of $y\in K^\circ$ is an isometry with respect to Thompson's metric, and hence $z\in\mathcal{M}(x,y)$ if and only if $P(y^{-\frac{1}{2}})z\in\mathcal{M}(P(y^{-\frac{1}{2}})x,e)$. Thus, without loss of generality, we may consider midpoints sets of the form $\mathcal{M}(x,e)$ where $x\in K^\circ$. 

The following lemma, which  is Exercise III.3 in \cite{FK} will be useful in the sequel. A proof can be found in \cite[Lemma 6.1]{LR}. 
\begin{lemma}\label{j-prod in-prod} Let $K^\circ$ be a symmetric cone. For $x,y\in K$ we have $\langle x,y\rangle=0$ if and only if $x\bullet y=0$.
\end{lemma}
Given an idempotent $c\in K$ we have the Peirce decomposition $$V=V(c,0)\oplus V(c,{\textstyle \frac{1}{2}})\oplus V(c,1)$$ where $V(c,\lambda)$ are the corresponding eigenspaces of the only possible eigenvalues $\lambda$ that the linear operator $L(c)$ can have, see \cite[Proposition III.1.3]{FK}. Although this is a direct sum of vector spaces, both components $V(c,0)$ and $V(c,1)$ are Jordan subalgebras \cite[Proposition IV.1.1]{FK}, and for $\lambda=0,1$ we will denote the cone of squares in $V(c,\lambda)$ by $K(c,\lambda)$. Regarding the midpoint sets, we are particularly interested in $V(c,0)$. Note that this subalgebra has $e-c$ as a unit. 

For $x\in K^\circ$ with spectral decomposition $x=\sum_{i=1}^k\lambda_i c_i$ we let 
\[
\mathcal{C}_x:=\left\{c_i\in\{c_1\ldots,c_k\}:\ \max\{\lambda_i,\lambda_i^{-1}\}=\max\{\lambda_+(x),\lambda_-(x)^{-1}\}\right\}.
\]
Note that after reordering the eigenvalues $\lambda_1<\lambda_2<\cdots<\lambda_k$ we have that  $\mathcal{C}_x\subseteq\{c_1,c_n\}$. Before we characterise the affine span of the midpoints set $\mathcal{M}(x,e)$ for $x\in K^\circ$, we first prove the following lemma.
\begin{lemma}\label{face2}Let $V$ be a Euclidean Jordan algebra with cone of squares $K$ and let $c\in K$ be an idempotent. Then $K(c,0)$ is a face of $K$ with relative interior $$K(c,0)^\circ=\mathrm{Inv}(V(c,0))\cap K(c,0),$$ where $\mathrm{Inv}(V(c,0))$ denotes the set of invertible elements in the subalgebra $V(c,0)$.
\end{lemma}
\begin{proof}Let $z\in K(c,0)$. If $\xi_1,\xi_2\in K$ and $0<t<1$ are such that $z=t\xi_1+(1-t)\xi_2$, then 
\[
0\leq t\inprod{c,\xi_1}+(1-t)\inprod{c,\xi_2}=\inprod{c,z}=0,
\] 
so $\xi_1,\xi_2\in K(c,0)$ by Lemma \ref{j-prod in-prod}, and hence $K(c,0)$ is a face of $K$. Note that the Jordan subalgebra $V(c,0)$ has unit $e-c\in K(c,0)$, since $(e-c)^2=e-c$. The fact that  $K(c,0)^\circ=\mathrm{Inv}(V(c,0))\cap K(c,0)$ now follows from \cite[Theorem III.2.1]{FK}.
\end{proof}

Note that $K(c,0)=F_{e-c}$. Indeed, if $x\in K(c,0)$, then $x\le n(e-c)$ for some $n\ge 1$, as $e-c$ is an order unit in $V(c,0)$. So, $K(c,0)\subseteq F_{e-c}$ 
by Lemma \ref{face}. Conversely, if $y\in F_{e-c}$, then $0\le_K y\le_K n(e-c)$ for some $n\ge 1$. It now follows that 
\[
0\leq\inprod{c,y}=\inprod{c,y}-\inprod{c,n(e-c)}=\inprod{c,y-n(e-c)}\leq 0,
\] 
and hence $y\in K(c,0)$ by Lemma \ref{j-prod in-prod}. 

We can prove the characterisation of the midpoints set in symmetric cones. 
\begin{theorem} \label{thm:4.3} Let $K^\circ$ be a symmetric cone. For $x\in K^\circ\setminus\{e\}$ let $\mathcal{C}_x$ be defined as above and put $c:=\sum_{c_i\in\mathcal{C}_x}c_i$. The affine span of $\mathcal{M}(x,e)$ satisfies 
$$\aff\mathcal{M}(x,e)=m_{xe}+V(c,0).$$
\end{theorem}
\begin{proof}
Let $x=\lambda_1 c_1+\cdots+\lambda_k c_k$ be the spectral decomposition of $x$ with $\lambda_1<\cdots<\lambda_k$. First suppose that $\mathcal{C}_x=\{c_k\}$. Note that $\lambda_k>1$, as $d_T(x,e) =\log \lambda_k>0$.  The endpoint 
$$x'=\frac{\lambda_k}{\lambda_k-1}e-\frac{1}{\lambda_k-1}x=\sum_{i=1}^{k-1}\frac{\lambda_k-\lambda_i}{\lambda_k-1}c_i,$$ where $e$ is between $x'$ and $x$, is in the relative interior of $K(c,0)$  by Lemma \ref{face2}, as it is invertible in $V(c,0)$ with respect to $e-c_k=c_1+\cdots+c_{k-1}$. The desired equality now follows from Theorem \ref{thm:3.2} since $K(c,0)$ is generating in $V(c,0)$. In the same way it can be shown that the assertion  holds if  $\mathcal{C}_x=\{c_1\}$. Finally, if $\mathcal{C}_x=\{c_1,c_k\}$, then $c=c_1+c_k$ and it follows from Theorem \ref{thm:3.2} that $$\aff \mathcal{M}(x,e)=m_{xe}+V(c_1,0)\cap V(c_k,0).$$ 
Clearly $V(c,0)\supseteq V(c_1,0)\cap V(c_k,0)$.
As $V(c_i,0)=\ker L(c_i)$ for $i=1,k$ and $K(c,0)\subseteq\ker L(c_1)\cap\ker L(c_k)$, we must have that $V(c,0)=V(c_1,0)\cap V(c_k,0)$, since $\Span K(c,0)= V(c,0)$.
\end{proof}
It follows from Theorem \ref{thm:4.3} that $\dim\mathrm{aff}(\mathcal{M}(x,e))=\dim V(c,0)$. If $V$ is a simple Euclidean Jordan algebra, i.e., $V$ has no non-trivial ideals, then for any two orthogonal primitive idempotents $c_1$ and $c_2$ in $V$ the dimension $$d:=\dim V(c_1,{\textstyle\frac{1}{2}})\cap V(c_2,{\textstyle\frac{1}{2}})$$ is independent of $c_1$ and $c_2$. In fact,  if $\mathrm{rank}(V)=r$, then $$n=r+{\textstyle\frac{d}{2}}r(r-1)$$ by \cite[Corollary IV.2.6]{FK}. 

Now let us decompose  $x$ with respect to a Jordan frame, for details see \cite[Theorem III.1.2]{FK}, where we might have different primitive idempotents corresponding to the same eigenvalues, so 
$$x=\lambda_1 c_1+\cdots+\lambda_r c_r.$$ 
We can  rearrange the eigenvalues in such a way that \begin{align}\label{frame rep}x=\sum_{j\in\mathcal{C}_x^c}\lambda_jc_j+\sum_{c_i\in\mathcal{C}_x}\lambda_i c_i.\end{align}  
It follows that for $k:=|\mathcal{C}_x^c|$  the dimensional formula $$\dim V(c,0)=\dim V(e-c,1)=k+{\textstyle\frac{d}{2}}k(k-1)$$ 
holds by \cite[Proposition IV.3.1]{FK}. This immediately gives the following corollary.
\begin{corollary}\label{affine dimension}Let $V$ be a simple $n$-dimensional  Euclidean Jordan algebra with cone of squares $K$ and rank $r>1$. If $x\in K^\circ$, then the affine dimension of the midpoint set satisfies $$\dim\aff(\mathcal{M}(x,e))=k+\frac{n-r}{r(r-1)}k(k-1).$$
\end{corollary}

As an example, let us  consider the Hermitian matrices $\mathbb{H}_n(\mathbb{C})$ and compute $\dim(\aff\mathcal{M}(A,I_n))$ for $A\in\Pi_n(\CC)^\circ$ and $I_n$ the $n\times n$ identity matrix. There exists an unitary matrix $U$ such that $UAU^{*}=D$ where $D$ is a diagonal matrix  with the eigenvalues on the diagonal arranged as in (\ref{frame rep}). Conjugating with $U$ is a linear automorphism of $\Pi_n(\CC)$, and $B\in\mathcal{M}(A,I_n)$ if and only if $U^{*}BU\in\mathcal{M}(D,I_n)$. So, to compute the dimension of $\aff\mathcal{M}(A,I_n)$ we may assume without loss of generality that $A$ is a diagonal matrix as described above. In that case the projection $C$  is of the form $$C=\sum_{i=k+1}^nE_{ii}=\begin{pmatrix}0&0\\0&I_{n-k}\end{pmatrix}$$ where $k=|\mathcal{C}_A^c|$. It is easily checked that $V(C,0)$ equals $$V(C,0)=\left\{\begin{pmatrix}A_k&0\\0&0\end{pmatrix}:\ A_k\in\mathbb{H}_k(\CC)\right\},$$ see \cite[p.\,63]{FK}, and  $\dim(V(C,0))=k^2$. Since $\mathrm{rank}(V)=n$ and $\dim(V)=n^2$, it follows from Corollary \ref{affine dimension} that $$\dim(\aff\mathcal{M}(A,I_n))=k+k(k-1)=k^2.$$

 \vspace{15mm} 
\hfill \begin{tabular}{l l }
  Bas Lemmens & Mark Roelands\\
 School of Mathematics, & School of Mathematics\\
 Statistics \& Actuarial Sciences  & Statistics \& Actuarial Sciences\\
 University of Kent & University of Kent\\
CT27NF, Canterbury & CT27NF, Canterbury \\
 United Kingdom & United Kingdom \\
 B. Lemmens@kent.ac.uk & mark.roelands@gmail.com\\

\end{tabular}

\end{document}